\documentclass[a4paper,11pt]{amsart}

\usepackage{amsmath,amssymb,amsthm}
\usepackage{graphicx}
\usepackage[all]{xy}
\usepackage{wrapfig}
\usepackage{pxfonts}

\DeclareMathAlphabet{\mathbbold}{U}{bbold}{m}{n}

\def\k{\mathbbold{k}}

\DeclareSymbolFont{rsfscript}{OMS}{rsfs}{m}{n}
\DeclareSymbolFontAlphabet{\mathrsfs}{rsfscript}

\DeclareFontFamily{OMS}{rsfs}{\skewchar\font'177}
\DeclareFontShape{OMS}{rsfs}{m}{n}{%
      <5> rsfs5
      <6> <7> rsfs7
      <8> <9> <10> rsfs10
      <10.95> <12> <14.4> <17.28> <20.74> <24.88> rsfs10
      }{}

\def\calA{\mathrsfs{A}}

\def\calC{\mathrsfs{C}}

\def\calF{\mathrsfs{F}}

\def\calH{\mathrsfs{H}}
\def\calI{\mathrsfs{I}}

\def\calM{\mathrsfs{M}}

\def\calP{\mathrsfs{P}}
\def\calQ{\mathrsfs{Q}}
\def\calR{\mathrsfs{R}}

\def\calV{\mathrsfs{V}}

\DeclareMathOperator{\id}{id}

\DeclareMathOperator{\Sym}{Sym}

\DeclareMathOperator{\Ord}{Ord_+}
\DeclareMathOperator{\Vect}{Vect}

\DeclareMathOperator{\st}{st}

\makeatletter

\theoremstyle{plain}

\newtheorem {theorem}{Theorem}
\newtheorem {lemma}{Lemma}
\newtheorem {corollary}{Corollary}

\newtheorem {proposition}{Proposition}

\theoremstyle{definition}

\newtheorem {definition}{Definition}
\newtheorem {remark}{Remark}
\newtheorem {example}{Example}

\begin{document}

\title{Shuffle algebras, homology,\\ and consecutive pattern avoidance}
\author{Vladimir Dotsenko}
\address{Mathematics Research Unit, University of Luxembourg, Campus Kirchberg, 6, Rue Richard Coudenhove-Kalergi, L-1359 Luxembourg, Grand Duchy of Luxembourg}
\email{vladimir.dotsenko@uni.lu}
\author{Anton Khoroshkin}
\address{Simons Center for Geometry and Physics,
Stony Brook University, Stony Brook, NY 11794-3636, USA 
	  and 
	 ITEP, 
	 Bolshaya Cheremushkinskaya 25, 
	 117259, Moscow, 
	 Russia}
\email{khorosh@itep.ru}



\thanks{The second author's research was supported by grants NSh3349.2012.2, RFBR-10-01-00836, RFBR-CNRS-10-01-93111, RFBR-CNRS-10-01-93113, and a grant Ministry of Education and Science of the Russian Federation under the contract 14.740.11.081.}

\begin{abstract}
Shuffle algebras are monoids for an unconvential monoidal category structure on graded vector spaces. We present two homological results on shuffle algebras with monomial relations, and use them to prove exact and asymptotic results on consecutive pattern avoidance in permutations. 
\end{abstract}

\maketitle

\section{Introduction}

The goal of this paper is two-fold. First of all, it is intended to develop some homological algebra tools for shuffle algebras defined by Maria Ronco \cite{Ronco} (called also \emph{permutads} in a recent paper \cite{LodayRonco}). Namely, our main result can be viewed as the computation of appropriate Tor groups for shuffle algebras with monomial relations (the case of non-monomial relations may be handled in a usual way by Gr\"obner bases and homological perturbation \cite{DK}). This generalises for the case of shuffle algebras a celebrated construction of Anick~\cite{Anick}.

On the other hand, our result has a transparent combinatorial meaning. Shuffle algebras with monomial relations have bases that can be naturally described via (generalised coloured) permutations avoiding given consecutive patterns. A permutation $\tau$ is said to occur in a permutation $\sigma$ as a consecutive pattern if there exists a subword of $\sigma$ which is order-isomorphic to~$\tau$. Free resolutions that we construct allow to give combinatorial formulae for inverses of the corresponding exponential generating functions. A simple example one can have in mind is as follows. Permutations avoiding the consecutive pattern $12$ are precisely the \emph{decreasing} permutations, there is exactly one such permutation of each length~$n$. ``On the dual level'', the space of generators of the corresponding free resolution is spanned by permutations where all the subwords of length two are order-isomorphic to~$12$, that is \emph{increasing} permutations. There is also exactly one such permutation of each length~$n$. This leads to the inversion formula
 $$
\sum_{n\ge0}\frac{t^n}{n!}=\frac{1}{1-t+\sum_{q\ge2}\frac{(-1)^q}{q!}t^q},
 $$ 
where one recognises an elementary formula
 $$
\exp(t)=\frac{1}{\exp(-t)}.
 $$
One particular application of our approach for longer patterns is a proof of a conjecture of Elizalde~\cite{Elizalde} on patterns without self-overlaps. When we prepared the first draft of this paper, we learned that this conjecture was independently proved by Adrian Duane and Jeffrey Remmel~\cite{DR} based on methods developed in~\cite{MR}.

The above example, as well as many similar ones, fits into a very simple combinatorial proof using the inclusion-exclusion principle. The combinatorial formalism for that is called the cluster method of Goulden and Jackson~\cite{GJ2,NZ}. However, the formulas provided on that way have many terms cancelling for somewhat trivial reasons. Unlike that, our approach gives formulas free from those trivial cancellations. Further progress in algorithmic and computational approaches to consecutive pattern avoidance is presented in recent preprints \cite{BNZ,N}. We also wish to mention a follow-up \cite{KS} to an earlier version of this paper showing the relevance of homological methods for studying consecituve patterns.

The paper is organised as follows. In Section \ref{sec:ShAlg} we give the definition of a shuffle algebra, and explain how shuffle algebras can be used to study consecutive pattern avoidance. Then, before constructing our free resolutions in full detail, we begin with exploring the low homological degrees in Section~\ref{sec:GSh}. It turns out that they can be used to obtain various asymptotic results on consecutive pattern avoidance, in the spirit of Golod--Shafarevich approach~\cite{GS}. We re-prove several results in that direction previous obtained by Elizalde~\cite{Eli1}, and derive various new ones. Finally in Section~\ref{sec:Anick}, we construct a free resolution of the trivial module over a shuffle algebra with monomial relations, and discuss applications of this resolution. A reader primarily interested in applications to combinatorics should refer to Sections \ref{sec:App1} and \ref{sec:App2}; though these sections contain refererences to results proved in more algebraic parts of the paper, they are close to being self-contained in all other respects.

All vector spaces throughout this work are defined over an arbitrary field~$\k$ of zero characteristic. We adopt the usual notation $[n]$ for the set $\{1,2,\ldots,n\}$. The group of permutations of a finite set~$I$ is denoted by $\Sym(I)$. In case $I=[n]$, we use a more concise notation~$S_n$ for the permutation group.

\subsection*{Acknowledgements. }We wish to thank Sergi Elizalde and Sergey Kitaev for their remarks about consecutive patterns, and to Jean-Louis Loday and Maria Ronco for useful conversations on shuffle algebras. We are also grateful to Dmitri Piontkovski who drew our attention to the fact that some of our techniques had previously been used in~\cite{P} in the case of associative algebras. The work on this paper started when the first author was working at, and the second author was visiting Dublin Institute for Advanced Studies; they express their deepest gratitude to all the staff there for their hospitality. 

\section{Shuffle algebras}\label{sec:ShAlg}

\subsection{Nonsymmetric collections and shuffle products}

In this section, we shall recall the definition of a shuffle algebra, as defined by Ronco in~\cite{Ronco} (see also the paper of Loday and Ronco~\cite{LodayRonco}). Our definitions and methods, though equivalent to the original definition of Ronco (and the subsequent definition of Loday and Ronco), are different, and rather follow the approach of~\cite{DK}.  

We denote by~$\Ord$ the category whose objects are finite ordered sets (with order-preserving bijections as morphisms). Also, we denote by $\Vect$ the category of vector spaces (with linear operators as morphisms).

\begin{definition}
\begin{enumerate}
\item A \emph{(nonsymmetric) collection} is a contravariant functor from the category~$\Ord$ to the category~$\Vect$.
\item Let $\calP$ and $\calQ$ be two nonsymmetric collections. Define their \emph{shuffle tensor product} $\calP\boxtimes\calQ$ by the formula
 $$
(\calP\boxtimes\calQ)(I):=\bigoplus_{J\sqcup K=I}\calP(J)\otimes\calQ(K),
 $$
where the sum is taken over all partitions of $I$ into two disjoint subsets $J$ and~$K$.
\end{enumerate}
\end{definition}

\begin{remark}
\begin{enumerate}
 \item Nonsymmetric collections are in one-to-one correspondence with (nonnegatively) graded vector spaces (for a functor $\calF$, the graded component $F_n$ of the corresponding graded vector space~$F$ is $\calF([n])$). However, the functorial definition makes the monoidal structure much easier to handle, with one exception: to \emph{define} a nonsymmetric collection, it is sufficient to define the spaces~$\calF([n])$, with all other spaces defined automatically because of functoriality. We shall use this observation many times throughout the paper.
 \item If we define the tensor product of two nonsymmetric collections by a similarly looking formula 
 $$
(\calP\otimes\calQ)(I):=\bigoplus_{J+K=I}\calP(J)\otimes\calQ(K),
 $$
where the sum is taken over all partitions of $I$ into two consecutive intervals $J$ and~$K$, this would indeed give the standard tensor product of graded vector spaces.
\end{enumerate}
\end{remark}

The following proposition is straightforward; we omit the proof.

\begin{proposition} The shuffle tensor product endows the category of nonsymmetric collections with a structure of a monoidal category. The unit object in each case is the functor $\calI$ which vanishes on all nonempty sets and is one-dimensional for the empty set.
\end{proposition}

The following proposition shows that the shuffle tensor product provides a ``categorification'' of the product of exponential generating functions in the same way as the usual tensor product provides a categorification of the product of ``normal'' generating functions.

\begin{proposition}
For a nonsymmetric collection $\calP$, let us define its exponential generating series~$f_\calP(t)$ as the power series $\sum_{n\ge0}\frac{\dim\calP([n])}{n!}t^n$.
Then we have 
\begin{equation}\label{eq:multiplicativity}
f_{\calP\boxtimes\calQ}(t)=f_\calP(t)\cdot f_\calQ(t). 
\end{equation}
\end{proposition}

\begin{proof}
Indeed, the number of ways to split $[n]$ into a disjoint union $[n]=J\sqcup K$ with $|J|=j$, $|K|=k$ is equal to
 $$
\binom{n}{j}=\frac{n!}{j!(n-j)!}=\frac{n!}{j!k!},
 $$ 
so 
 $$
\dim((\calP\boxtimes\calQ)([n]))=\sum_{0\le j\le n}\frac{n!}{j!k!}\dim(\calP([j]))\dim(\calQ([k])),
 $$
and the result follows.
\end{proof}

\subsection{Shuffle algebras}
\begin{definition}
A \emph{shuffle (associative) algebra} is a monoid in the category of nonsymmetric collections with the monoidal structure given by the shuffle tensor product. 
\end{definition}

In other words, to define a shuffle algebra structure on a nonsymmetric collection $\calA$, one has to define the \emph{structure maps}
 $$
\mu_{J,K}\colon\calA(J)\otimes\calA(K)\to\calA(J\sqcup K)
 $$
satisfying the obvious associativity conditions. 

\begin{remark}
Shuffle algebras are closely related to twisted associative algebras (see, e.g., \cite{Stover}), namely, they are in the same relationship with them as shuffle operads are with symmetric operads. Also, the category of shuffle algebras admits an embedding into the category of shuffle operads, and this embedding is behind some of the constructions of this paper. We shall not discuss these topics in detail here. 
\end{remark}

\begin{example}\label{ex:AssociativeAlgebra}
Every graded associative algebra $V$ gives rise to a shuffle algebra~$\widetilde{V}$ with $\widetilde{V}(I)=V_{|I|}$, where for every partition $I=J\sqcup K$ the corresponding product map 
 $$
\mu_{J,K}\colon\widetilde{V}(J)\otimes\widetilde{V}(K)=V_{|J|}\otimes V_{|K|}\to V_{|J|+|K|}=\widetilde{V}(I)
 $$ 
is given by the product in~$V$. 
\end{example}

\begin{example}\label{ex:FreeAlgebra}
Consider the shuffle algebra $\calA_{MR}$ with $\calA_{MR}(I)=\k\Sym(I)$, where for every partition $I=J\sqcup K$ the corresponding product map   
 $$
\mu_{J,K}\colon\calA_{MR}(J)\otimes\calA_{MR}(K)=\k\Sym(J)\otimes\k\Sym(K)\to\k\Sym(I)=\calA_{MR}(I)
 $$ 
is somewhat tautological: the product of two permutations is the permutation of~$I=J\sqcup K$ obtained from the respective permutations of $J$ and $K$ by concatenation. 
\end{example}

As shown in~\cite{Ronco}, the algebra from the previous example is isomorphic to the free shuffle algebra with one generator of degree~$1$. This shuffle algebra gives a refinement of (the underlying graded algebra of) the Malvenuto--Reutenauer Hopf algebra of permutations~\cite{MRHopf}. Many other Hopf algebras of combinatorial nature, e.g. the Hopf algebra of quasi-symmetric functions, the Hopf algebra of parking functions, the Hopf algebra of set partitions etc. (for definitions, see \cite{LR2010} and references therein) are shuffle algebras as well, with the associative product being the sum over all possible shuffle products. 

Let us give the combinatorial construction of a free algebra generated by a given nonsymmetric collection. Let $\calM$ be a nonsymmetric collection with $\calM(\varnothing)=\{0\}$, and let $\mathsf{B}$ be a nonsymmetric collection of finite ordered sets (that is, a functor from the category $\Ord$ to itself) such that for every ordered set~$I$ the set $\mathsf{B}(I)$ is a basis of $\calM(I)$. We shall describe a nonsymmetric collection of finite ordered sets that will form a bases in components of the free shuffle algebra. By definition, elements of $\mathbb{B}(I)$ correspond to the following combinatorial data:
\begin{enumerate}
 \item an ordered partition of~$I$ into subsets, $I=\bigsqcup_{j=1}^mI_j$;
 \item a ``monomial'' $c_1c_2\ldots c_m$ with $c_j\in\mathsf{B}(I_j)$ for every $j=1,\ldots,m$. 
\end{enumerate}
The shuffle product $\mu_{J,K}$ concatenates both the ordered partitions and the monomials. 

Note that if we assume that $\calM(I)=\{0\}$ for $|I|\ne1$ and $\dim\calM(I)=1$ for $|I|=1$, we see that every subset $I_j$ has to consist of one element, and therefore any ordered partition that contributes is just a permutation (and the monomials do not carry additional information, capturing the lengths of the permutations). Therefore, we recover the free algebra with one generator of degree~$1$ from the example~\ref{ex:FreeAlgebra} above. 

The following proposition is straightforward.

\begin{proposition}
The collection $F\langle\calM\rangle$ with $F\langle\calM\rangle(I)=\mathop{\mathrm{span}}\mathbb{B}(I)$ is (isomorphic to) the free shuffle algebra generated by~$\calM$.
\end{proposition}

\subsubsection{Shuffle ideals and modules}
Since shuffle algebras are monoids in a monoidal category, the usual definitions of ideals, quotients, modules, etc. can be immediately given in this context. To make the article self-contained, we present them here. All shuffle algebras in this paper are assumed to be \emph{connected}, that is having $\k$ as the empty set component.

\begin{definition}
Let $\calA$ be a shuffle algebra with the product $\mu\colon\calA\boxtimes\calA\to\calA$.
\begin{itemize}
 \item A right module over~$\calA$ is a nonsymmetric collection $\calM$ together with a structure map $\gamma\colon\calM\boxtimes\calA\to\calM$ satisfying the associativity condition 
 $$
\gamma(\gamma\boxtimes\id_\calA)=\gamma(\id_\calM\boxtimes\mu). 
 $$
 \item The trivial right module over~$\calA$ is the collection $\calI$ which has $\k$ as the empty set component and zero for all other components, where the only nonzero part of the structure map is $$\calI(\varnothing)\otimes\calA(\varnothing)=\k\otimes\k\simeq\k=\calI(\varnothing).$$
 \item The regular right module over~$\calA$ is the collection~$\calA$ itself, with the structure map $\gamma=\mu$.
 \item A right ideal of $\calA$ is a subcollection of the regular right module which is closed under the structure map.
 \item For a subcollection $\calR$ of $\calA$, the right ideal $(\calR)$ generated by $\calR$ is the minimal right ideal of $\calA$ that contains~$\calR$. 
 \item The free right module over~$\calA$ generated by the nonsymmetric collection $\calV$ is the collection $\calV\boxtimes\calA$ with the structure map $\gamma=\id_\calV\boxtimes\mu$. A free module is said to be finitely generated if all components $\calV(I)$ are finite-dimensional, and moreover they vanish for $|I|$ sufficiently large.
\end{itemize}
The respective definitions of left modules, left ideals, bimodules, and two-sided ideals are completely analogous.  
\end{definition}

The following is an example of how graded associative algebras can be presented as shuffle algebras with generators and relations, i.e. as quotients of free shuffle algebras.

\begin{example}\label{ex:QuadraticQuotient}
Let us take the algebra~$\calA_{MR}$ discussed in Example~\ref{ex:FreeAlgebra}, and compute its quotient modulo the two-sided ideal generated by the difference $12-21\in\k S_2$. This quotient is isomorphic to the algebra $\widetilde{V}$ from Example~\ref{ex:AssociativeAlgebra} with $V=\k[x]$.
\end{example}

\subsubsection{Consecutive patterns}

In this section, we shall explain how our definitions are related to the combinatorial concept of consecutive pattern avoidance. 

Let us recall some definitions and notation. To every sequence $s$ of length $k$ consisting of $k$ distinct numbers, we assign a permutation $\st(s)$ of length $k$ called the standardization of~$s$; it is uniquely determined by the condition that $s_i< s_j$ if and only if $\st(s)_i<\st(s)_j$. For example, $\st(153)=132$. In other words, $\st(s)$ is a permutation whose relative order of entries is the same as that of~$s$. We say that a permutation $\sigma$ of length~$n$ avoids the given permutation $\tau$ of length~$j$ as a consecutive pattern if for each $j<n-i+1$ we have $\st(\sigma_i\sigma_{i+1}\ldots\sigma_{i+j-1})\ne\tau$, otherwise we say that $\sigma$ contains $\tau$ as a consecutive pattern. Throughout this paper, we only deal with consecutive patterns, so the word ``consecutive'' will be omitted. For historical information on pattern avoidance in general and the state-of-art for consecutive patterns, we refer the reader to~\cite{KitHist,Stein}. 

The central question arising in the theory of pattern avoidance is that of enumeration of permutations of given length that avoid the given set of forbidden patterns~$P$ or, more generally, contain exactly~$l$ occurrences of patterns from~$P$. This question naturally leads to the following equivalence relations. Two sets of patterns $P$ and $P'$ are said to be Wilf equivalent (notation: $P\simeq_W P'$) if for every $n$, the number of $P$-avoiding permutations of length~$n$ is equal to the number of $P'$-avoiding permutations of length~$n$. This notion (in the case of one pattern) is due to Wilf~\cite{Wilf}. More generally, $P$ and $P'$ are said to be equivalent (notation: $P\simeq P'$) if for every $n$ and every $k\ge 0$, the number of permutations of length~$n$ with~$k$ occurrences of patterns from $P$ is equal to the number of permutations of length~$n$ with $k$ occurrences of patterns from~$P'$.
 
While studying the equivalence classes of patterns, sometimes it is possible to replace the set of forbidden patterns by a Wilf equivalent one with fewer patterns in it. Namely, we have a partial ordering on the set of all permutations (of all possible lengths): $\tau<\sigma$ if $\sigma$ contains $\tau$ as a consecutive pattern. Given a set $P$ of ``forbidden'' patterns, to enumerate the permutations avoiding all patterns from~$P$, we may assume that $P$ is an antichain with respect to this partial ordering. Indeed, ignoring all patterns from~$P$ that contain a smaller forbidden subpattern does not change the set of $P$-avoiding permutations. Therefore, further on we shall assume that forbidden patterns do indeed form an antichain.

\subsubsection{Shuffle algebras and consecutive patterns}

The following result, however simple, provides a bridge between algebra and combinatorics, defining for each forbidden set $P$ of patterns a shuffle algebra whose exponential generating series is precisely the exponential generating function for the numbers of permutations avoiding~$P$. Let us denote by $a^P_n$ the number of permutations of length~$n$ that avoid all patterns from~$P$, and by $g_P(t)$ the corresponding exponential generating function,
 $$
g_P(t):=1+\sum_{n\ge1}\frac{a^P_n}{n!}t^n.
 $$

\begin{theorem}\label{ShuffleAndPatterns}
For every set~$P$ of forbidden patterns, let us define the shuffle algebra $\calA^P_{MR}$ as the quotient of the algebra $\calA_{MR}$ modulo the two-sided ideal generated by all patterns from $P$. Then the (classes of) permutations avoiding all patterns from $P$ form a basis of the quotient. Consequently,
 $$
f_{\calA^P_{MR}}(t)=g_P(t).
 $$
\end{theorem}

\begin{proof}
Since the products in $\calA_{MR}$ are defined via concatenations, it is clear that the ideal generated by~$P$ consists precisely of permutations containing patterns from~$P$. This means that we may identify classes in the quotient $\calA_{MR}/(P)$ with permutations avoiding patterns from~$P$. We shall use this identification throughout the paper.
\end{proof}

A similar result for the free shuffle algebra with more than one generator provides technical tools to deal with pattern avoidance in coloured permutations \cite{Ma}, and more general consecutive pattern avoidance where, for instance, each occurrence of a rise of length~$2$ may or may not be coloured etc. We shall not discuss the corresponding applications in this paper, but want to draw the reader's attention that all our methods generalise immediately to those settings. 

\subsubsection{Modules over the associative operad}

This short section is intended for those readers whose intuition, as it is for us, comes from the operad theory. Essentially, it re-tells the shuffle algebra approach in a slightly different way, explaining also the place for classical pattern avoidance in the story (recall that classical patterns are those occurring as subsequences rather than as factors in permutations).

Studying varieties of algebras, that is, algebras satisfying certain identities, goes back to works of Specht~\cite{Specht}. The notions of $T$-ideals and $T$-spaces formalize the ways to derive identities from one another. One natural way to study identities is to define an analogue of a Gr\"obner basis for an ideal of identities. This approach is taken in works of Latyshev \cite{Lat1,Lat2} who suggested a combinatorial approach to study associative algebras with additional identities via standard bases of the corresponding $T$-spaces. His approach can be described as follows. For each ``$T$-space'' (in other words, right ideal in the associative operad), he defines a version of a Gr\"obner basis; such a basis would allow to study arbitrary relations via monomials avoiding certain patterns. Here, for once, by a pattern we mean a classical pattern (its occurrence does not have to be as a consecutive subword, but rather a subsequence). This approach has a slight disadvantage. Namely, even though the actual Gr\"obner bases of relations are expected to be finite (at least, the famous result of Kemer~\cite{Kemer} states that in principle there exists a finite set of generating identities), they are difficult to compute, as there is no algorithm comparable to the one due to Buchberger in the associative algebra case~\cite{Ufn}. Remarkably, this trouble disappears if we study left ideals in the associative operad. In terms of combinatorics, studying left ideals also has a very clear meaning: the corresponding notion of divisibility corresponds to \emph{consecutive} pattern avoidance! For consecutive patterns, the intuition of \cite{DKRes,DK} for Gr\"obner bases and resolutions applies directly, and it turns out to be possible to describe the relevant resolutions explicitly, in fact the level of complexity here being closer to the case of associative algebras than to the case of operads.

\subsection{Shuffle homological algebra}

One of the central concepts of homological algebra is that of a derived functor. Computing derived functors relies on being able to construct ``nice'' (free, projective, injective etc.) resolutions of objects to which we want to apply our derived functors. The category of objects of primary interest to us is the category of left modules over the given shuffle algebra~$\calA$, and a typical functor we want to derive is ``shuffle torsion groups'', i.e. the derived functor of the shuffle tensor product over~$A$ with a given module, e.g. with the trivial right module. This paper is focused on combinatorial applications of shuffle algebras, so in the view of Theorem~\ref{ShuffleAndPatterns} the shuffle algebras of main interest for us are quotients of the algebra~$\calA_{MR}$ modulo the ideal generated by several patterns. In the following sections, we shall present two results of homological algebra for such shuffle algebras, and derive from these results various statements on enumerative combinatorics of consecutive patterns. One technical result that we shall be using to translate between the two languages is the following standard statement on Euler characteristics, applied to nonsymmetric collections.

\begin{proposition}\label{EulerChar}
Let 
 $$
\ldots\to\calC_n\to\ldots\to\calC_2\to\calC_1\to\calC_0
 $$
be a chain complex of nonsymmetric collections with homology groups $\calH_0$, $\calH_1$, \ldots, $\calH_n$, \ldots Then we have 
 $$
f_{\calC_0}(t)-f_{\calC_1}(t)+\ldots+(-1)^nf_{\calC_n}(t)+\ldots=
f_{\calH_0}(t)-f_{\calH_1}(t)+\ldots+(-1)^nf_{\calH_n}(t)+\ldots,
 $$
provided that the sum on the left and on the right makes sense (for every integer $l$ only finitely many summands have nonzero coefficients of~$t^l$). 
\end{proposition}

\section{Golod--Shafarevich type complex and its applications}\label{sec:GSh}

\subsection{Golod--Shafarevich type inequality}

In this section, we shall exhibit a very simple application of homological algebra philosophy to combinatorics, mimicking the idea used by Golod and Shafarevich in 1960s in their study of class field tower \cite{GS} which has been used a lot in algebra and combinatorics since then, see e.g. \cite{P} and the later papers \cite{GShApp1,GShApp2,GShApp3,GShApp4}. Namely, we shall constructing the low homological degree part of the minimal resolition of the trivial right module over the shuffle algebra $\calA^P_{MR}$ by free modules. More precisely, we shall prove the following theorem.  

\begin{theorem}\label{GShForShuffle}
Let $\calV$ be the one-dimensional space generating the free algebra $\calA_{MR}$, and $\calP$ be the subcollection of $\calA_{MR}$ spanned by forbidden patterns. There exists a chain complex
\begin{equation}\label{GShComplex}
\calP\boxtimes\calA^P_{MR}\to\calV\boxtimes\calA^P_{MR}\to\calA^P_{MR}\to\calI\to0 
\end{equation}
which is exact everywhere except for the leftmost term. 
\end{theorem}

\begin{proof}
The boundary maps are as follows: 
\begin{enumerate}
 \item $\calA^P_{MR}\to\calI$ is the augmentation, mapping all permutations of positive length to zero,
 \item $\calV\boxtimes\calA^P_{MR}\to\calA^P_{MR}$ is the product in the algebra $\calA^P_{MR}$ (generated by~$\calV$),
 \item $\calP\boxtimes\calA^P_{MR}\to\calV\boxtimes\calA^P_{MR}$ is the composition of the inclusion 
 $$
\calP\boxtimes\calA^P_{MR}\to\calV\boxtimes\calA_{MR}\boxtimes\calA^P_{MR}
 $$ 
(which exists because $\calV$ generates the algebra $\calA_{MR}$), the projection 
 $$
\calV\boxtimes\calA_{MR}\boxtimes\calA^P_{MR}\to\calV\boxtimes\calA^P_{MR}\boxtimes\calA^P_{MR},
 $$ 
and the product in the algebra $\calA^P_{MR}$.
\end{enumerate}
The exactness of this complex in the terms $\calI$ and $\calA^P_{MR}$ is obvious. Let us show the exactness in the term $\calV\boxtimes\calA^P_{MR}$. Since the relations of the algebra~$\calA^P_{MR}$ are monomial, the kernel of the boundary map is spanned by ``monomials'' $j\boxtimes\rho$ with $j$ being an element of degree~$1$ and $\rho$ being a permutation avoiding patterns from~$P$. Such an element belongs to the kernel of the boundary map if $j\rho=0$, therefore, $j\rho$ contains a pattern from $P$. Such a pattern has to be an initial segment of $j\rho$, otherwise $\rho$ would contain a pattern from~$P$ itself. This instantly implies that our element is in the image of the boundary map from~$\calP\boxtimes\calA^P_{MR}$. 
\end{proof}

\begin{corollary}\label{GShSeries}
Let $P=\sqcup_{n\ge2}P_n$ be a collection of forbidden consecitive patterns in permutations. Then the following coefficient-wise inequality holds
\begin{equation}\label{GSh}
\left(1-t+\sum_{k\ge2}\frac{|P_k|}{k!}t^k\right)g_P(t)\ge1. 
\end{equation}
\end{corollary}

\begin{proof}
Let us denote by~$\calH_3$ the only nontrivial piece of homology of the chain complex~\eqref{GShComplex}.
Computing Euler characteristics according to Proposition~\ref{EulerChar}, we see that
 $$
f_\calI(t)-f_{\calA^P_{MR}}(t)+f_{\calV}(t)f_{\calA^P_{MR}}(t)-f_{\calP}(t)f_{\calA^P_{MR}}(t)=-f_{\calH_3}(t),
 $$
or
 $$
1-g_P(t)+tg_P(t)-\left(\sum_{k\ge2}\frac{|P_k|}{k!}t^k\right)g_P(t)=-f_{\calH_3}(t),
 $$
which implies
 $$
\left(1-t+\sum_{k\ge2}\frac{|P_k|}{k!}t^k\right)g_P(t)=1+f_{\calH_3}(t)\ge1.
 $$
\end{proof}

\subsection{Applications to consecutive pattern avoidance}\label{sec:App1}

One key application of Corollary~\ref{GShSeries} is contained in the following

\begin{corollary}
Assume that the power series $f(t)=1-t+\sum_{k\ge2}\frac{|P_k|}{k!}t^k$ has a root $\alpha>0$. Then
 $$
a^P_n\ge\alpha^{-n}n!.
 $$
\end{corollary}

\begin{proof}
Let $\sum_{l\ge0}b_lt^l:=\frac{1}{1-t+\sum_{k\ge2}\frac{|P_k|}{k!}t^k}$, so that $b_0=1$ and
 $$
b_n-b_{n-1}+\sum_{k=2}^n\frac{|P_k|}{k!}b_{n-k}=0.
 $$
Let us prove by induction that $b_n\ge \alpha^{-1}b_{n-1}$. Indeed, for $n=1$ this statement is obvious ($\alpha\ge1$ because otherwise $f(\alpha)$ is evidently positive), and for $n>1$ we note that by the induction hypothesis $b_{n-1}\ge\alpha^{1-k}b_{n-k}$, so
\begin{multline*}
b_n=b_{n-1}-\sum_{k=2}^n\frac{|P_k|}{k!}b_{n-k}\ge b_{n-1}-\sum_{k=2}^n\frac{|P_k|}{k!}\alpha^{k-1}b_{n-1}\ge\\ \ge
b_{n-1}-\sum_{k\ge2}\frac{|P_k|}{k!}\alpha^{k-1}b_{n-1}=\alpha^{-1}b_{n-1}\left(\alpha-\sum_{k\ge2}\frac{|P_k|}{k!}\alpha^{k}\right)=\alpha^{-1}b_{n-1},
\end{multline*}
which proves the step of induction. Therefore 
 $$
b_n\ge\alpha^{-n},
 $$
and the series $\frac{1}{1-t+\sum_{k\ge2}\frac{|P_k|}{k!}t^k}$ has positive coefficients. Hence multiplying the inequality \eqref{GSh} by that series preserves the inequality, and we obtain
 $$
g_P(t)\ge\frac{1}{1-t+\sum_{k\ge2}\frac{|P_k|}{k!}t^k},
 $$
so
 $$
\frac{a^P_n}{n!}\ge\alpha^{-n},
 $$
which completes the proof.
\end{proof}

Using the above corollary, one can obtain good asymptotic results on enumeration of permutations avoiding the given set of consecutive patterns, thus rediscovering a result of Elizalde \cite{Eli1} in the case of one pattern, but also recovering some much stronger results. Let use give several examples.

\begin{corollary}
The number of permutations of length~$n$ avoiding the given single pattern~$\tau$ of length~$k$ is at least~$\alpha_k^{-n}n!$, where $\alpha_k$ is the smallest positive root of the equation $1-t+\frac{t^k}{k!}=0$. (For example, $\alpha_4\approx1.050800769$, $\alpha_5\approx1.008702295$, $\alpha_6\approx1.001400601$.)
\end{corollary}

\begin{corollary}
Let the set of forbidden patterns $P$ contain one pattern of each length~$l\ge4$. Then the number of permutations of length~$n$ avoiding~$P$ is at least $\alpha^{-n}n!$, where $\alpha\approx1.068290263$ is the root of the equation $e^t-2t-\frac{t^2}{2}-\frac{t^3}{6}=0$. In particular, there are infinitely many permutations avoiding $P$ regardless of the actual choice of patterns in~$P$.
\end{corollary}

\begin{proof}
In this case, $|P_n|=1$ for all $n\ge4$, so 
 $$
1-t+\sum_{k\ge2}\frac{|P_k|}{k!}t^k=1-t+e^t-1-t-\frac{t^2}{2}-\frac{t^3}{6}=e^t-2t-\frac{t^2}{2}-\frac{t^3}{6}.
 $$
\end{proof}

\section{Anick-type resolution and its applications}\label{sec:Anick}

\subsection{Anick-type resolution}

In this section, we shall explain how to extend the complex we constructed above to a resolution of the trivial right module by free $\calA^P_{MR}$-modules. The generators of those modules are defined combinatorially. Once the set $P$ of forbidden patterns is fixed, we define, for each nonnegative integer $q$, the notion of a $q$-chain and the tail of a given $q$-chain associated to $P$ inductively as follows. 

\begin{itemize}
\item[-] the empty permutation is a $0$-chain on the empty set, it coincides with its tail;
\item[-] the only permutation of a one element set $I$ is a $1$-chain on $I$, it also coincides with its tail;
\item[-] each $q$-chain is a permutation $\sigma$ represented as a concatenation $\sigma'\tau$, where $\tau$ is the tail of $\sigma$, and $\sigma'$ is a $(q-1)$-chain on its underlying set;
\item[-] if we denote by~$\tau'$ the tail of $\sigma'$ in the above representation, then $\tau'\tau$ contains exactly one occurrence of a pattern from~$P$, and this occurrence is a terminal segment of~$\tau'\tau$.
\end{itemize}

The way we define the chains here is slightly different from the original approach of Anick~\cite{Anick}; the reader familiar with the excellent textbook of Ufnarovskii~\cite{Ufn} will rather notice similarities with the approach to Anick resolution adopted there.

Informally, a $q$-chain is a ``minimal'' way to form a permutation by linking together $(q-1)$ prohibited patterns. The word ``minimal'' is justified by the following

\begin{lemma}
No proper beginning of a $q$-chain is a $q$-chain.
\end{lemma}

\begin{proof}
We shall prove it by induction on~$q$, the basis of induction ($q=0,1,2$) being obvious.

Assume that there is a $q$-chain $\sigma=\sigma'\tau$ which has a proper beginning that is a $q$-chain as well, so $\tau=\mu\nu$, and $\sigma'\mu$ is a $q$-chain. By the induction hypothesis, no proper beginning of a $(q-1)$-chain is a $(q-1)$-chain which implies that $\mu$ is the tail of the $q$-chain~$\sigma'\mu$. However, this immediately shows that (in the notation of the definition of chains and tails above) $\tau'\tau$ contains at least two different occurrences of patterns from~$P$, which is a contradiction.  
\end{proof}

One more fact about chains that makes the above definition more transparent is that, even though we defined a chain as a permutation together with a factorisation, in fact the factorisation carries no additional information:

\begin{lemma}\label{ChainsArePermutations}
If $\sigma$ is a $q$-chain, the way to link $(q-1)$ patterns from~$P$ to one another to form~$\sigma$ is unique.
\end{lemma}

\begin{proof}
Assume that there are two ways to link $q$ patterns to form~$\sigma$. Obviously, for each $m<q$, the endpoints of the $m^\text{th}$ (from left to right) patterns in these two linkages should coincide, otherwise we shall find an $m$-chain whose proper beginning is an $m$-chain as well which is not the case by the previous lemma. Once we know that the endpoints of the $m^\text{th}$ patterns are the same, the beginnings have to be the same because $P$ is assumed to be an antichain (and so patterns from~$P$ cannot be contained in one another). 
\end{proof}

Let us give some examples clarifying the notion of a chain. For example, if $P=\{12\}$, the only $q$-chain for each $q$ is $12\ldots q$, while if $P=\{123\}$, we can easily see that $123$ is the only $1$-chain, and $1234$ is the only $2$-chain, but $12345$ is not a $2$-chain because it starts from a $2$-chain $1234$, and it is not a $3$-chain because in the only way to cover this permutation by three copies of our pattern the first and the third occurences overlap: 
 $$
\underbrace{1\quad 
 2\quad \makebox[0pt][l]{$\overbrace{\phantom{2\quad 3\quad 4}}$}3}\quad 4\quad 5.
 $$

\begin{theorem}\label{AnickForShuffle}
Denote by $\calC_q$ the subcollection of the free algebra $\calA_{MR}$ spanned by all $q$-chains. There exists a chain complex 
\begin{equation}\label{AnickComplex}
\ldots\calC_q\boxtimes\calA^P_{MR}\to\calC_{q-1}\boxtimes\calA^P_{MR}\to\ldots\to\calC_1\boxtimes\calA^P_{MR}\to\calA^P_{MR}\to\calI\to0,
\end{equation}
which is exact in every term.
\end{theorem}

Note that this result is a direct generalisation of the one of Theorem \ref{GShForShuffle} since $\calC_2=\calP$, $\calC_1=\calV$.

\begin{proof}
The boundary map $\calC_q\boxtimes\calA^P_{MR}\to\calC_{q-1}\boxtimes\calA^P_{MR}$ is defined as a composition of the inclusion 
 $$
\calC_q\boxtimes\calA^P_{MR}\to\calC_{q-1}\boxtimes\calA_{MR}\boxtimes\calA^P_{MR}
 $$ 
(which exists because we can factorise a $q$-chain as a product of a $(q-1)$-chain and a tail), the projection 
 $$
\calC_{q-1}\boxtimes\calA_{MR}\boxtimes\calA^P_{MR}\to\calC_{q-1}\boxtimes\calA^P_{MR}\boxtimes\calA^P_{MR},
 $$ 
and the product in the algebra~$\calA^P_{MR}$. 

Let us prove the exactness of this complex in the term $\calC_q\boxtimes\calA^P_{MR}$. Since the relations of the algebra~$\calA^P_{MR}$ are monomial, the kernel of the boundary map is spanned by ``monomials'' $\sigma\otimes\rho$ with $\sigma$ is a $q$-chain and $\rho$ is a permutation avoiding patterns from~$P$. Such an element belongs to the kernel of the boundary map if $\sigma'\otimes\tau\rho=0$, where $\tau$ is the tail of $\sigma$, and $\sigma=\sigma'\tau$. Therefore, $\tau\rho$ contains a pattern from $P$. Since $\rho$ avoids patterns from~$P$, this means that there exists a decomposition $\rho=\rho'\rho''$ such that $\tau\rho'$ contains a pattern from~$P$ as its terminal segment, and this is the only occurrence of a pattern from~$P$ in $\tau\rho'$ (take for $\rho'$ the smallest initial segment of~$\rho$ with this property). This immediately implies that $\sigma\rho'$ is a $(q+1)$-chain with the tail~$\rho'$, so our element is the image under the boundary map of the element $\sigma\rho'\otimes\rho''$.
\end{proof}

Let us denote by $c_{n,q}$ the number of $q$-chains which are permutations of length~$n$. 
\begin{corollary}\label{AnickSeries}
We have 
\begin{equation}\label{Inverse}
g_P(t)=\frac{1}{1-t+\sum\limits_{q\ge2,n\ge1}\frac{(-1)^qc_{n,q}}{n!}t^n}.
\end{equation}
\end{corollary}

\begin{proof}
Computing Euler characteristics according to Proposition~\ref{EulerChar}, we see that
 $$
f_\calI(t)-f_{\calA^P_{MR}}(t)+f_{\calC_1}(t)f_{\calA^P_{MR}}(t)-f_{\calC_2}(t)f_{\calA^P_{MR}}(t)+f_{\calC_3}(t)f_{\calA^P_{MR}}(t)-\ldots=0
 $$
or
 $$
1-g_P(t)+tg_P(t)-\left(\sum_{k\ge2}\frac{(-1)^q c_{n,q}}{n!}t^n\right)g_P(t)=0,
 $$
which implies
 $$
g_P(t)=\frac{1}{1-t+\sum\limits_{q\ge2,n\ge1}\frac{(-1)^qc_{n,q}}{n!}t^n},
 $$
as required.
\end{proof}

Let us give a simple example in which both the left hand side and the right hand side of the equation~\eqref{Inverse} can be easily computed (we already mentioned it in the introduction). Let $P$ consist of a single pattern~$12$. Then, for each $q$ we have one $q$-chain $12\ldots q$ of length~$q$. Also, for every $m$ the only permutation of length~$m$ avoiding $12$ is $m(m-1)\ldots21$. Therefore, the inversion formula above becomes
 $$
\sum_{n\ge0}\frac{t^n}{n!}=\frac{1}{1-t+\sum_{q\ge2}\frac{(-1)^q}{q!}t^q},
 $$ 
and we recognise the well known formula
 $$
\exp(t)\exp(-t)=1.
 $$

Equation \eqref{Inverse} carries a striking resemblance with a celebrated result of Goulden and Jackson \cite{GJ2} expressing the inverses of generating functions for consecutive pattern avoidance in terms of \emph{clusters}:
\begin{equation}\label{ClusterInverse}
g_P(t)=\frac{1}{1-t+\sum\limits_{q\ge2,n\ge1}\frac{(-1)^qcl_{n,q}}{n!}t^n},
\end{equation}
where $cl_{n,q}$ is the number of $q$-clusters of length~$n$. A $q$-cluster is, roughly speaking, an indecomposable covering of a permutation by patterns from the forbidden set~$P$, but, unlike chains, without any minimality condition. As a consequence, the number of chains is potentially much smaller than the number of clusters, and our result is a strengthening of the result of Goulden and Jackson. A good way to think of it is to say that many ``obvious'' cancellations happen in the cluster formula \eqref{ClusterInverse}, and our approach takes care of these ''obvious'' cancellations.\footnote{We want to note, however, that our approach can be used to prove the cluster inversion formula too, if one adapts the method of \cite{DKRes} for constructing free resolutions.} For example, we already saw that for $P=\{123\}$ the permutation $12345$ is not a chain. However, it can be covered by two copies of $123$ as well as by three copies of $123$, and these coverings give it a structure of a $2$-cluster and a $3$-cluster respectively. The contributions of these two clusters in \eqref{ClusterInverse} occur with opposite signs, and the total contribution of this permutation is equal to zero, exactly as \eqref{Inverse} suggests. Among the applications below, for some of the examples it does not really matter if we are dealing with chains or clusters, whereas for other ones chains give more compact formulas.

\subsection{Applications to consecutive pattern avoidance}\label{sec:App2}

Before moving on to particular results, let us state a general remark. Our results suggest that the class of power series that contains all inverses of pattern avoidance enumerators is related to some nice combinatorics. Results of Elizalde and Noy~\cite{EN} that we re-prove below describe some of these series as solutions to particular differential equations. Our formulas for other cases we considered can be rewritten as more complicated functional equations. What can be said about other series of that sort? So far we have not able to describe a reasonable class of series that cover all of these. A wild  guess is that all these series satisfy algebraic differential equations, that is, if $f(x)$ is such a series, then $P(x,f(x),f'(x),\ldots,f^{(d)}(x))=0$ for some nonzero polynomial~$P(x,t_0,t_1,\ldots,t_d)$. 

\subsubsection{Patterns without self-overlaps, linking schemes, and posets}

In this section, we shall enumerate chains in one particular case, namely, the case of an arbitrary pattern without self-overlaps, which will allow us to prove a conjecture of Elizalde~\cite{Elizalde}. In fact, in this case chains coincide with clusters, so one could refer to results of Goulden and Jackson instead of Theorem~\ref{AnickForShuffle}.

\begin{definition}
A pattern $\tau$ is said to have no self-overlaps if every permutation of length at most $2m-2$ has at most one occurrence of~$\tau$. (Clearly, there always exist permutations of length~$2m-1$ with two occurrences of $\tau$.)
\end{definition}

For example, the pattern $132$ is of that form: clearly, we can only link it with itself using the last entry. A more general example studied in~\cite{EN} is $12\ldots a\ \tau\ (a+1)\in\Sigma_n$, where $a+1<n$, and $\tau$ is an arbitrary permutation of the numbers $a+2,\ldots,n$. 

For a pattern~$\tau$ without self-overlaps, there exists a simple way to reformulate the enumeration problem for chains in terms of total orderings on posets. The first author used this method in~\cite{DVJ} in a similar setting, dealing with tree monomials in the free shuffle operad. To a $q$-chain~$\sigma$ obtained by linking $q-1$ copies of~$\tau$, let us assign a ``linking scheme'' of the shape that we expect, replacing each entry in~$\sigma$ by the symbol~$\bullet$ (a bullet), and marking the segments of consecutive bullets that are ``traces'' of (occurrences of) $\tau$. For example, for the pattern~$1243$ and $4$-chains we get
 $$
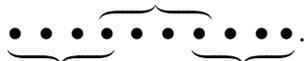
\underbrace{\bullet\,\,\bullet\,\,\bullet\,\,\,
\makebox[0pt][l]{$\overbrace{\phantom{\bullet\,\,\bullet\,\,\bullet\,\,\bullet\,\,}}$}\bullet}\,\,\bullet\,\,\bullet\,\,
\underbrace{\bullet\,\,\bullet\,\,\bullet\,\,\bullet}.
 $$ 
For such a linking scheme, let us define a partial ordering on bullets as follows: for each $j$, we equip the $j^\text{th}$ trace of $\tau$ with a total ordering identical to the ordering of the corresponding entries of~$\tau$. Let us denote by $\Pi_{q,\tau}$ the thus defined poset.

\begin{example}
Let us take the linking scheme above, and replace bullets by letters, to make it easier to distinguish between different bullets: 
 $$
\underbrace{a\,\,b\,\,c\,\,\,
\makebox[0pt][l]{$\overbrace{\phantom{d\,\,e\,\,f\,\,g\,\,}}$}d}\,\,e\,\,f\,\,
\underbrace{g\,\,h\,\,i\,\,j}.
 $$ 
Then the orderings inherited from~$1243$ are $a<b<d<c$, $d<e<g<f$, and $g<h<i<j$, so we obtain the poset $\Pi_{4,1243}$
 $$ 
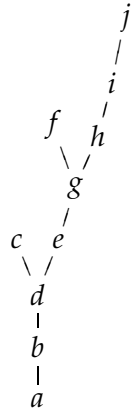
\xygraph{!{<0mm,0mm>;<1.5mm,0mm>:<0mm,1.5mm>::}
!{(47.5263,31.7229)*+{a}}="0"
!{(47.5263,36.4093)*+{b}}="1"
!{(47.5263,41.0957)*+{d}}="2"
!{(45.6784,45.7822)*+{c}}="3"
!{(49.3742,45.7822)*+{e}}="4"
!{(50.8250,50.4686)*+{g}}="5"
!{(48.8709,56.1550)*+{f}}="6"
!{(52.8792,55.1550)*+{h}}="7"
!{(54.1106,59.8415)*+{i}}="8"
!{(55.2824,65.6279)*+{j}}="9"
"0"-"1"
"1"-"2"
"2"-"3"
"2"-"4"
"4"-"5"
"5"-"6"
"5"-"7"
"7"-"8"
"8"-"9"
}
 $$
(the covering relation of the poset is, as usual, represented by edges; $v$ is covered by $w$ if $w$ is the top vertex of the corresponding edge).
\end{example}

The following proposition is obvious. 
\begin{proposition}
The set of $q$-chains for $P=\{\tau\}$, where $\tau$ has no self-overlaps, is in one-to-one correspondence with the set of all total orderings on posets $\Pi_{q,\tau}$. 
\end{proposition}

Now we shall see how this approach can be applied in some cases. 

\subsubsection{Case of the pattern~$12\ldots a\ \tau\ (a+1)$}

Let $a<m$, and let $12\ldots a\ \tau\ (a+1)$ be a permutation of length~$m+1$ which starts with the increasing run $1,2,\ldots,a$, followed by some permutation~$\tau$ of $(a+2),\ldots,m+1$, followed by the number~$(a+1)$. Clearly, this pattern has no self-overlaps, so to enumerate chains we may count total orderings of posets. Note that every $(q+1)$-chain for $q\ge0$ is of length $q(m+1)-(q-1)=qm+1$. 

\begin{proposition}
For $P=\{12\ldots a\ \tau\ (a+1)\}$, the number of~$(q+1)$-chains is equal to
 $$
\prod_{j=1}^q\binom{jm-a}{m-a}.
 $$
\end{proposition}

\begin{proof}
This proof serves us as a starting example of how to use posets to study chains. The poset $\Pi_{q,\tau}$ in this case looks like a tree of height $m+1$ with the only branch growing on the height~$a+1$, this branch being of length~$m+1$ and having a smaller branch growing at the distance $a+1$ from the starting point, etc. (An example of such a poset for the case of the permutation~$1243$ with $a=2$, $m=3$ is given above.) To extend such a partial ordering to a total ordering, we should make the lowest $a+1$ elements for such a tree the smallest elements $1,2,\ldots,a+1$ of the resulting ordering. Then, there are $\binom{qm-a}{m-a}$ ways to choose $(m+1)-(a+1)=m-a$ remaining elements forming the stem of our tree, and we are left with the same question for a smaller tree, where we may proceed by induction.
\end{proof}

\begin{corollary}[see \cite{EN,KitPOP} for $t=0$]\label{1non}
For $a<m$, the multiplicative inverse of the generating function $g_P(t)$ of permutations avoiding $12\ldots a\ \tau\ (a+1)\in S_{m+1}$ is given by the formula
\begin{equation}\label{EN}
1-t-\sum_{q\ge1}\frac{(-1)^{q+1} t^{qm+1}}{(qm+1)!}\prod_{j=1}^q \binom{jm-a}{m-a}. 
\end{equation}
In particular, all these patterns, for different~$\tau$, are Wilf equivalent to each other.
\end{corollary}

Except for the case of the pattern $123\simeq_W 321$, this covers all patterns of length~$3$, because $132\simeq_W 312\simeq_W 231\simeq_W 213$ (the equivalence provided by either reversing the order of entries in the pattern from the left to the right, or reversing the relative order of entries in the pattern). We shall deal with the pattern~$123$ and, more generally, $12\ldots a$, in further sections.

\subsubsection{Case of one arbitrary pattern without self-overlaps}

Generalising the previous result, let us consider an arbitrary pattern $\tau$ of length~$m+1$ without self-overlaps. For such a pattern, every $(q+1)$-chain for $q\ge0$ is still of length~$qm+1$. The following result was conjectured in~\cite{Elizalde}, where it was proved in some particular cases. Another proof in the general case was, as we discovered after the first version of this paper got in circulation, obtained by Adrian Duane and Jeffrey Remmel~\cite{DR}; it is based on entirely different techniques developed in~\cite{MR}.

\begin{theorem}\label{wilf}
For a pattern $\tau$ of length~$m+1$ without self-overlaps, the number of permutations of length~$n$ with $k$ occurrences of $\tau$ depends only on $n$, $k$, $m$, $\tau(1)$, and $\tau(m+1)$. In other words, two non-self-overlapping permutations of length~$m+1$ are equivalent if their first and last entries are the same.
\end{theorem}

\begin{proof}
Since for patterns without self-overlaps clusters coincide with chains, and cluster inversion can be used to count permutations with a given number of occurrences of forbidden patterns \cite{GJ2}, it is enough to show that the number of $(q+1)$-chains depends only on the first and the last entry of~$\tau$. This result is also very easy to derive using posets. To make formulas compact, let us put $a=\tau(1)-1$ and $b=\tau(m+1)-1$. The poset $\Pi_{q,\tau}$ whose total orderings enumerate $q$-chains is obtained from $q$ totally ordered sets of cardinality~$m+1$ as follows: the element $a+1$ of the second set is identified with the element~$b+1$ of the first set, the element $a+1$ of the third set is identified with the element $b+1$ of the second set, etc. Clearly, this poset depends only on $m$, $a$, and $b$. 

The actual number of $q$-chains in this case can be computed as follows. Let us denote by $f_k(p)$ the number of $q$-chains $\sigma$ whose first element is~$p+1$. Then it is easy to see that the following recurrence relation holds (here we assume, without the loss of generality, that $a<b$):
\begin{equation}\label{recur}
f_k(p)=\sum_{q}\binom{p}{a}\binom{km-q}{m-b}\binom{q-p-1}{b-a-1}f_{k-1}(q-b). 
\end{equation}
Indeed, if we denote $q+1=\sigma(m+1)$, there are $\binom{p}{a}$ ways to choose elements less than~$p+1$ in the first pattern in the chain, $\binom{km-q}{m-b}$ ways to choose elements greater than $\sigma(m+1)$ there, $\binom{q-p-1}{b-a-1}$ to fill the space between these elements, and $f_{k-1}(q-b)$ ways to choose the remaining $(k-1)$-chain. 
\end{proof}

\begin{example}
Theorem~\ref{wilf} shows that the two patterns $23154$ and $21534$ are equivalent to each other. Computing the first ten cluster numbers and inverting the corresponding series, we get the first ten entries $1$, $1$, $2$, $6$, $24$, $119$, $708$, $4914$, $38976$, $347776$ of the sequence counting permutations that avoid either of them. 
\end{example}

\subsubsection{Case of one pattern of length~$4$}

Let us now consider the case of a single pattern of length~$4$. The equivalence classes of these are as follows (see~\cite{Elizalde}):
\begin{enumerate}
 \item[I.] $1234\simeq4321$
 \item[II.] $2413\simeq3142$
 \item[III.] $2143\simeq3412$
 \item[IV.] $1324\simeq4231$
 \item[V.] $1423\simeq3241\simeq4132\simeq2314$
 \item[VI.] $1342\simeq2431\simeq4213\simeq3124\simeq1432\simeq2341\simeq4123\simeq3214$
 \item[VII.] $1243\simeq3421\simeq4321\simeq2134$ 
\end{enumerate}

The case~I will be considered later. In each of the cases VI and VII, the pattern has no self-overlaps, so Corollary~\ref{1non} applies.  

A very special feature of all patterns of length~$4$ (except for the case~I) is that they have self-overlaps of length at most~$2$, so however we try to link several patterns together, it will be  automatically true that only neighbours overlap. Moreover, even if we are dealing with a pattern~$\tau$ with self-overlaps, every labelling of a linking scheme that is compatible with ordering of each of the patterns gives a genuine chain. Assume that $\gamma$ is a linking scheme for $q$ copies of~$\tau$. By induction, we may assume that the linking scheme provided by the first $q-1$ traces of~$\tau$ only gives chains, and we only need to check the chain condition for the terminal segment, for which the statement follows from the fact that if two patterns of length~$4$ overlap by a segment of length~$1$ or~$2$, then every pattern of length~$4$ overlapping with the both of them overlaps with at least one of them by a segment of length~$3$. Guided by this observation, we compute all the exponential generating functions of consecutive patterm avoidance. Since in this case chains coincide with clusters, our results can be easily adapted for enumeration of permutations with a given number of occurrences of a given pattern.

\begin{theorem}\label{1324}
The numbers $c_{n,l}$ for the pattern $1324$ satisfy the recurrence relations
\begin{equation}\label{rec-catalan}
c_{n,l}=\sum_{4\le 2k+2\le n}\frac{1}{k+1}\binom{2k}{k}c_{n-2k-1,l-k}  
\end{equation}
with initial conditions $c_{1,l}=\delta_{0,l}$ (the Kronecker delta symbol), $c_{2,l}=0$, $c_{3,l}=0$. 
Consequently, the generating function for avoidance of $1324$ is 
 $$
\left(1-t-\sum_{n\ge2,l\ge1}\frac{c_{n,l}t^n(-1)^l}{n!}\right)^{-1}.
 $$
\end{theorem}

\begin{proof}
As we discussed above, counting chains is reduced to counting total orderings of the corresponding posets. Let us assume that the first $k+1$ patterns have two-element overlaps, and the following overlap involves just one element. For a chain $\sigma=a_1a_2\ldots a_{2k+1}a_{2k+2}a_{2k+3}\ldots$, this means that
\begin{equation}\label{catalan}
a_1<a_3<a_2<a_4, a_3<a_5<a_4<a_6, \ldots, a_{2k-1}<a_{2k+1}<a_{2k}<a_{2k+2}, 
\end{equation}
that $\{a_1,\ldots,a_{2k+2}\}=\{1,\ldots,2k+2\}$, and that $\st(a_{2k+2}a_{2k+3}\ldots)$ is an $(l-k)$-chain. To prove~\eqref{rec-catalan}, we notice that the number of permutations $a_1a_2\ldots a_{2k+2}$ of $\{1,\ldots,2k+2\}$ for which the conditions~\eqref{catalan} are satisfied is given by the number of standard Young tableaux of size $2\times k$:  clearly, $a_1=1$, $a_{2k+2}=2k+2$, and 
 $$
a_2,a_3,a_4,\ldots,a_{2k+1} \quad\leftrightarrow\quad  
\begin{tabular}{|p{0.8cm}|p{0.8cm}|p{0.8cm}|p{0.8cm}|p{0.8cm}|}
\hline
$a_3$& $a_5$ &$a_7$ &\ldots&$a_{2k+1}$\\
\hline
$a_2$& $a_4$& $a_6$ &\ldots&$a_{2k}$\\
\hline
\end{tabular}
 $$
gives a bijection with standard Young tableaux. The number of such tableaux is equal to the Catalan number $\frac{1}{k+1}\binom{2k}{k}$ (see, for example~\cite{Stan}), and the recurrence relation~\eqref{rec-catalan} follows.
\end{proof}

\begin{example}
Computing the first ten of those numbers and inverting the corresponding series, we get the first ten entries $1$, $1$, $2$, $6$, $23$, $110$, $632$, $4229$, $32337$, $278204$ of the sequence which is indeed counting permutations that avoid~$1324$ (A113228 in~\cite{oeis}). 
\end{example}

\begin{theorem}\label{1423}
The numbers $c_{n,l}$ for the pattern $1423$ satisfy the recurrence relations
\begin{equation}\label{shuffling-rec}
c_{n,l}=\sum_{4\le 2k+2\le n}\binom{n-k-2}{k}c_{n-2k-1,l-k} 
\end{equation}
with initial conditions $c_{1,l}=\delta_{0,l}$, $c_{2,l}=0$, $c_{3,l}=0$. Consequently, 
the generating function for avoidance of $1423$ is 
 $$
\left(1-t-\sum_{n\ge2,l\ge1}\frac{c_{n,l}(-1)^l}{n!}\right)^{-1}.
 $$
\end{theorem}

\begin{proof}
Similarly to the proof of Theorem~\ref{1324}, counting chains is reduced to counting total orderings of the corresponding posets.
Let us assume that the first $k+1$ patterns have two-element overlaps, and the following overlap involves just one element. For a chain $\sigma=a_1a_2\ldots a_{2k+1}a_{2k+2}a_{2k+3}\ldots$, this means that
\begin{equation}\label{shuffling1}
a_1<a_3<a_4<a_2, a_3<a_5<a_6<a_4, \ldots, a_{2k-1}<a_{2k+1}<a_{2k+2}<a_{2k}, 
\end{equation}
so 
\begin{equation}\label{shuffling}
a_1<a_3<\ldots<a_{2k-1}<a_{2k+1}<a_{2k+2}<a_{2k}<\ldots<a_4<a_2, 
\end{equation}
$\{a_1,a_3,\ldots,a_{2k+1}\}=\{1,2,\ldots,k+1\}$, $a_{2k+2}=k+2$,  and $\st(a_{2k+2}a_{2k+3}\ldots)$ is an $(l-k)$-chain. To prove~\eqref{shuffling-rec}, we notice that the number of ways to distribute numbers between the increasing sequence~\eqref{shuffling} and the $(l-k)$-chain $\st(a_{2k+2}a_{2k+3}\ldots)$ is equal to the number of way to choose the $k$ numbers  
$a_{2k},\ldots,a_2$. The latter is clearly the binomial coefficient $\binom{n-k-2}{k}$, and the recurrence relation~\eqref{shuffling-rec} follows.
\end{proof}

\begin{example}
Computing the first ten of those numbers and inverting the corresponding series, we get the first ten entries $1$, $1$, $2$, $6$, $23$, $110$, $631$, $4218$, $32221$, $276896$ of the sequence counting permutations that avoid~$1423$. 
\end{example}

\begin{theorem}\label{2143}
The numbers $c_{n,l}$ for the pattern $2143$ satisfy the recurrence relations
 $$
 c_{n,l}=\sum_{2\le p<n-2}c_{n,l}(p),
 $$
where the numbers $c_{n,l}(p)$ satisfy the recurrence relations
\begin{equation}\label{2143-rec}
c_{n,l}(p)=\sum_{4\le2k+2\le q\le n}\binom{q-p-1}{2k-2}(p-1)(n-q)c_{n-2k-1,l-k}(q-2k)  
\end{equation}
with initial conditions $c_{1,l}(p)=\delta_{0,l}\delta_{1,p}$, $c_{2,l}(p)=0$, $c_{3,l}(p)=0$. 
Consequently, the generating function for avoidance of $2143$ is 
 $$
\left(1-t-\sum_{n\ge2,l\ge1}\frac{c_{n,l}t^n(-1)^l}{n!}\right)^{-1}.
 $$
\end{theorem}

\begin{proof}
Similarly to the proof of Theorem~\ref{1324}, counting chains is reduced to counting total orderings of the corresponding posets. Let $c_{n,l}(p)$ be the number of $l$-chains $\sigma$ of length~$n$ with $\sigma(1)=p$.
Let us assume that the first $k+1$ copies of $2143$ in $\sigma$ have two-element overlaps, and the following overlap involves just one element. For a chain $\sigma=a_1a_2\ldots a_{2k+1}a_{2k+2}a_{2k+3}\ldots$, this means that
\begin{equation}\label{monotone1}
a_2<a_1<a_4<a_3, a_4<a_3<a_6<a_5, \ldots, a_{2k}<a_{2k-1}<a_{2k+2}<a_{2k+1}, 
\end{equation}
so 
\begin{equation}\label{monotone}
a_2<a_1<a_4<a_3\ldots<a_{2k}<a_{2k-1}<a_{2k+2}<a_{2k+1}, 
\end{equation}
and $\st(a_{2k+2}a_{2k+3}\ldots)$ is an $(l-k)$-chain. Assume that $a_1=p$. To prove~\eqref{2143-rec}, we notice that if $a_{2k+2}=q$, then there are $\binom{q-p-1}{2k-2}$ ways to pick the numbers $a_3,\ldots,a_{2k}$, $p-1$ ways to pick~$a_2$, $(n-q)$ ways to pick $a_{2k+1}$, and $c_{n-2k-1,l-k}$ ways to pick the remaining $(l-k)$-chain (where the entry $q$ is the $(q-2k)^\text{th}$ biggest). This completes the proof.
\end{proof}

\begin{example}
Computing the first ten of those numbers and inverting the corresponding series, we get the first ten entries $1$, $1$, $2$, $6$, $23$, $110$, $631$, $4223$, $32301$, $277962$ of the sequence counting permutations that avoid~$2143$. 
\end{example}

In the last remaining case $2413\simeq3142$ (II in the list above), we have no trick like above that would simplify the computations, so we shall use the most general strategy for chain enumeration, which allows to compute the chain numbers rather fast (polynomially in~$n$) for all sets of forbidden patterns. There is an obvious similarity with the approach of Kitaev and Mansour in~\cite{KitManMult}. 

\begin{theorem}\label{2413}
The numbers $c_{n,l}$ for the pattern $2413$ are given by the formulae
 $$
 c_{n,l}=\sum_{1<p<q-1<n}c_{n,l}(p,q),
 $$
where the numbers $c_{n,l}(p,q)$ satisfy the recurrence relations
\begin{multline}\label{2413-rec}
c_{n,l}(p,q)=\sum_{r<p<s<q}c_{n-2,l-1}(r,s-1)+\\+\sum_{p<r<s<q}(p-1)c_{n-3,l-1}(r-1,s-1)+\sum_{p<r<q<s}(p-1)c_{n-3,l-1}(r-1,s-2)    
\end{multline}
with initial conditions $c_{2,l}(p,q)=0$, $c_{3,l}(p,q)=0$, $c_{4,l}(p,q)=\delta_{l,1}\delta_{p,2}\delta_{q,4}$. 
Consequently, the generating function for avoidance of $2143$ is 
 $$
\left(1-t-\sum_{n\ge2,l\ge1}\frac{c_{n,l}t^n(-1)^l}{n!}\right)^{-1}.
 $$
\end{theorem}

\begin{proof}
This statement is straightforward. Indeed, let us consider an $n$-chain $\sigma=a_1a_2a_3\ldots$. The first pattern in that chain intersects with its neighbour by either two or one elements. In the first case, we have $a_3<a_2<a_4<a_1$, so if we fix $a_1$ and $a_2$, and forget about them, we are left with an $(n-1)$-chain, and we should sum over all choices of $a_3$ and $a_4$ for its first entries. If, on the contrary the first overlap uses just one element, then there are $(a_1-1)$ choices for $a_3$, and we should distinguish between the cases $a_5>a_2$ and $a_5<a_2$: in the first case $a_5$ is the $(a_5-1)^\text{st}$ biggest in the remaining cluster, while in the second case it is the $(a_5-2)^\text{nd}$ biggest.
\end{proof}

\begin{example}
Computing the first ten of those numbers and inverting the corresponding series, we get the first ten entries $1$, $1$, $2$, $6$, $23$, $110$, $632$, $4237$, $32465$, $279828$ of the sequence counting permutations that avoid~$2413$. 
\end{example}

\subsubsection{Case of two patterns~$\{132,231\}$}

\begin{theorem}\label{132and231}
The number $c_{n,l}$ for $P=\{132,231\}$ is not equal to zero only for $n=2l+1$, and in this case is equal to $E_{2l+1}$, the tangent number~\cite{Stan}, so the generating function for avoidance of $\{132,231\}$ is 
\begin{equation}\label{tan}
\left(1-\tanh{t}\right)^{-1}. 
\end{equation}
\end{theorem}

\begin{proof}
This pair of patterns has no self-overlaps at all (both for a pattern with itself, and two patterns with each other), so every linking scheme clearly provides only chains. Clearly, chains are nothing but ``up--down'' permutations, that is permutations $a_1a_2\ldots a_{2l}a_{2l+1}$ for which $$a_1<a_2>a_3<a_4>\ldots<a_{2l}>a_{2l+1}.$$ It is well known that the number of such permutations is equal to the tangent number. 
\end{proof}

\subsubsection{Case of the pattern~$12\ldots k$}

The case we consider in this section is the case of the single pattern $12\ldots k$, which marks increasing runs of length~$k$ in permutations. The enumeration result in this case is well known, however, we want to show that it can also be obtained as a direct application of our results.

\begin{theorem}[\cite{EN,GJ1,KitPOP}]\label{rise}
The multiplicative inverse of the exponential generating function for patterns avoiding $12\ldots k$ is given by the formula
\begin{equation}\label{GJ1}
\sum_{q\ge0}\frac{x^{kq}}{(kq)!}-\sum_{q\ge0}\frac{x^{kq+1}}{(kq+1)!}.
\end{equation} 
\end{theorem}

\begin{proof}
Indeed, $q$-chains for $q\ge2$ are as follows:
\begin{itemize}
 \item[-] the only $2$-chain is $12\ldots k$;
 \item[-] the only $3$-chain is $12\ldots(k+1)$;
 \item[-] the only $4$-chain is $12\ldots(2k)$;
 \item[-] the only $5$-chain is $12\ldots(2k+1)$;
 \item[-] \ldots
 \item[-] the only $(2l)$-chain is $12\ldots(kl)$;
 \item[-] the only $(2l+1)$-chain is $12\ldots(kl+1)$;
 \item[-] \ldots
\end{itemize}
\end{proof}

\subsubsection{Case of the pattern~$\lambda(\lambda+m)\ldots(\lambda+(k-1)m)$}

The result of this section gives one way to somewhat generalise both Theorem~\ref{wilf} and Theorem~\ref{rise}. Let $\lambda$ be a pattern of length~$m$ without self-overlaps. Denote by~$\lambda+j$ the permutation of numbers $\{j+1,\ldots,j+m\}$ obtained by adding $j$ to each each entry of~$\lambda$. Let $\tau=\tau_{k,\lambda}=\lambda(\lambda+m+1)\ldots(\lambda+(k-1)m)$ be the ``ordered sum'' of $k$ copies of~$\lambda$. 

\begin{theorem}\label{ConcatenationOfSeveralNonoverlapping}
The number of permutations of length~$n$ avoiding $\tau$ depends only on $n$, $m$, $\tau(1)$, $\tau(m)$, and~$k$. In other words, for two non-self-overlapping patterns of length~$m$ the corresponding $k$-fold ordered sums are Wilf equivalent if their first and last entries are the same.
\end{theorem}

\begin{proof}
For the $k$-fold ordered sum of a pattern without self-overlaps, it is very easy to exhibit the linking schemes that actually give rise to chains. Such a linking scheme is a genuine mixture of linking schemes for patterns without self-overlaps and linking schemes for the pattern $12\ldots k$. Namely, for each $l\ge2$ there is one basic ``building block'', a linking scheme modelled on the $l$-chains  
\begin{itemize}
 \item[-] $\lambda(\lambda+m)\ldots(\lambda+(k-1)m)$ for $l=2$,
 \item[-] $\lambda(\lambda+m)\ldots(\lambda+(k-1)m)(\lambda+km)$ for $l=3$,
 \item[-] $\lambda(\lambda+m)\ldots(\lambda+(2k-2)m)(\lambda+(2k-1)m)$ for $l=4$,
 \item[-] $\lambda(\lambda+m)\ldots(\lambda+(2k-1)m)(\lambda+2km)$ for $l=5$,
 \item[-] \ldots
 \item[-] $\lambda(\lambda+m)\ldots(\lambda+(pk-2)m)(\lambda+(pk-1)m)$ for $l=2p$,
 \item[-] $\lambda(\lambda+m)\ldots(\lambda+(pk-1)m)(\lambda+pkm)$ for $l=2p+1$,
 \item[-] \ldots
\end{itemize}
and every linking scheme producing a chain is a linkage of several building blocks like that overlapping only by one element. The poset defined by such a linking scheme obviously depends only on the first and the last element of $\tau$ but not on the relative order of other elements. The corresponding recurrence relations can easily be derived from this description as well.
\end{proof}

\subsubsection{Case of the pattern~$12\ldots k$ and a pattern without self-overlaps}

This section gives another way to somewhat generalise both Theorem~\ref{wilf} and Theorem~\ref{rise}. Let $\lambda$ be a pattern of length~$m$ without self-overlaps. We shall study the enumeration problem for avoidance of $P_{\lambda,k}=\{\lambda,12\ldots k\}$. Let us introduce several parameters important for enumeration. Denote by $l_I(\lambda)$ the length of the maximal initial segment of $\lambda$ which is an increasing rise, and by~$l_T(\lambda)$ the length of the maximal terminal segment of $\lambda$ which is an increasing rise. Since we always assume patterns of~$P$ to not contain one another, and we assume $\lambda$ to have no self-overlaps, we conclude that $l_I(\lambda),l_T(\lambda)<k$, $l_I(\lambda)+l_T(\lambda)<m$ and $\min(l_I(\lambda),l_T(\lambda))=1$. 

\begin{theorem}\label{NonoverlappingAndRise}
The number of permutations of length~$n$ avoiding $P_{\lambda,k}$ depends only on $m$, $\lambda(1)$, $\lambda(m)$, $l_I(\lambda)$, $l_T(\lambda)$, and~$k$. In particular, if we adjoin to two non-self-overlapping patterns $\lambda_1$ and $\lambda_2$ of the same length~$m$ an increasing rise of length~$k$, the corresponding two-element sets are Wilf equivalent if the first and last entries, and the lengths of the initial and terminal increasing rises of $\lambda_1$ and $\lambda_2$ are the same.
\end{theorem}

\begin{proof}
Note that both reversing the direction in which we read permutations (left-to-right becomes right-to-left) and reversing the order of entries (increasing becomes decreasing) in all permutations considered preserve Wilf classes, and doing both these changes keeps the permutation $12\ldots k$ intact, we may assume that $l_T(\lambda)=1$. 

It is easy to exhibit the linking schemes that actually give rise to chains. Basically, there are two basic types of ``building blocks'' for the linking schemes: a linking scheme modelled on a single copy of $\lambda$ and linking schemes modelled on chains for a single pattern $12\ldots k$, as in the proof of Theorem~\ref{rise}. There is no freedom in linking copies of~$\lambda$ together: since $\lambda$ has no self-overlaps, two copies of $\lambda$ may only overlap by a single element. Since we assume that $l_T(\lambda)=1$, we conclude that an occurrence of $\lambda$ can only overlap with a building block coming from an overlap of several rises by a single element as well. For an overlap of several rises followed by an occurrence of~$\lambda$ the situation is different. Namely, if we are talking about the scheme modelled on the $(2l)$-chain $12\ldots(kl)$, it should overlap with the following copy of $\lambda$ by the initial increasing rise of that copy, i.e. by the first $l_I(\lambda)$ elements (since no proper beginning of a $q$-chain may be a $q$-chain). However, for the scheme modelled on the $(2l+1)$-chain $12\ldots(kl+1)$, it should overlap with the following copy of $\lambda$ by a single element (since only neighbouring patterns in a chain may overlap). Similarly to the proof of Theorem~\ref{wilf}, the posets defined by such linking schemes are completely determined by the first and the last entry of $\lambda$, and the lengths of its initial and terminal increasing rises. 
\end{proof}

\bibliographystyle{amsplain}

\begin{thebibliography}{10}

\bibitem{Anick}
David~J. Anick, \emph{On the homology of associative algebras}, Trans. Amer.
  Math. Soc. \textbf{296} (1986), no.~2, 641--659.

\bibitem{BNZ} 
Andrew Baxter, Brian Nakamura, and Doron Zeilberger, \emph{Automatic generation of theorems and proofs on enumerating consecutive-Wilf classes},
Preprint~\texttt{arXiv:1101.3949}.

\bibitem{GShApp1} 
Jason P.~Bell and Teow Goh, \emph{Exponential lower bounds for the number of words of uniform length avoiding a pattern},
Inform. and Comput. \textbf{205} (2007), no.~9, 1295--1306. 

\bibitem{GShApp2} 
Jason P.~Bell and Lance W.~Small, \emph{A question of Kaplansky}, J. Algebra \textbf{258} (2002), no.~1, 386--388.

\bibitem{Claesson}
Anders Claesson, \emph{Generalized pattern avoidance}, European J. Combin.
  \textbf{22} (2001), no.~7, 961--971.

\bibitem{DKRes}
Vladimir Dotsenko and Anton Khoroshkin, \emph{Quillen homology via {G}r\"obner
  bases}, Preprint \texttt{arXiv:1203.5053}.

\bibitem{DK}
\bysame, \emph{Gr{\"o}bner bases for operads}, Duke Math. J. \textbf{153} (2010), Number 2, 363--396.

\bibitem{DVJ}
Vladimir Dotsenko and Mikael {Vejdemo Johansson}, \emph{Implementing
  {G}r\"obner bases for operads}, S\'eminaires et Congr\`es, \textbf{26} (2011), 77--98.

\bibitem{DR} Adrian Duane and Jeffrey Remmel, \emph{Minimal overlapping patterns in colored permutations},
Electronic J. Comb. \textbf{18(2)}, 2011, article P25, 38~pp.  

\bibitem{Eli1}
Sergi Elizalde, \emph{Asymptotic enumeration of permutations avoiding generalized patterns},
Adv. in Appl. Math. \textbf{36} (2006), no. 2, 138--155. 

\bibitem{Elizalde}
Sergi Elizalde, \emph{Consecutive patterns and statistics on restricted
  permutations}, Ph.D. thesis, Universitat Polit\`ecnica de Catalunya, 2004.

\bibitem{EN}
Sergi Elizalde and Marc Noy, \emph{Consecutive patterns in permutations}, Adv.
  in Appl. Math. \textbf{30} (2003), no.~1-2, 110--125, Formal power series and
  algebraic combinatorics (Scottsdale, AZ, 2001).

\bibitem{GShApp3} 
Pavel Etingof and Victor Ginzburg, \emph{Noncommutative complete intersections and matrix integrals},
Pure Appl. Math. Q. \textbf{3} (2007), no. 1, 107--151. 

\bibitem{GS} Evgeny S.~Golod and Igor R.~Shafarevich, \emph{On the class field tower}. 
Izv. Akad. Nauk SSSR Ser. Mat. \textbf{28} (1964), 261--272. 

\bibitem{GJ2}
Ian~P. Goulden and David~M. Jackson, \emph{An inversion theorem for cluster
  decompositions of sequences with distinguished subsequences}, J. London Math.
  Soc. (2) \textbf{20} (1979), no.~3, 567--576.

\bibitem{GJ1}
\bysame, \emph{Combinatorial enumeration}, John Wiley \& Sons Inc., New York,
  1983, With a foreword by Gian-Carlo Rota, Wiley-Interscience Series in
  Discrete Mathematics.

\bibitem{Kemer}
Alexander~R. Kemer, \emph{Solution of the problem as to whether associative
  algebras have a finite basis of identities}, Soviet Math. Dokl. \textbf{37}
  (1988), no.~1, 60--64.

\bibitem{KS} Anton Khoroshkin and Boris Shapiro, \emph{Using homological duality in consecutive pattern avoidance},
Electronic J. Comb. \textbf{18(2)}, 2011, article P9, 17~pp.  

\bibitem{KitPOP}
Sergey Kitaev, \emph{Partially ordered generalized patterns}, Discrete Math.
  \textbf{298} (2005), no.~1-3, 212--229.

\bibitem{KitHist}
Sergey Kitaev and Toufik Mansour, \emph{A survey of certain pattern problems},
  Preprint, 2003.

\bibitem{KitManMult}
\bysame, \emph{On multi-avoidance of generalized patterns}, Ars Combin.
  \textbf{76} (2005), 321--350.

\bibitem{Lat1}
Victor~N. Latyshev, \emph{A general version of standard basis and its
  application to {$T$}-ideals}, Acta Appl. Math. \textbf{85} (2005), no.~1-3,
  219--223.

\bibitem{Lat2}
\bysame, \emph{Combinatorial generators of multilinear polynomial identities},
  J. Math. Sci. (N. Y.) \textbf{149} (2008), no.~2, 1107--1112.

\bibitem{LR2010} 
Jean-Louis Loday and Maria Ronco, Combinatorial Hopf algebras, Clay Math. Proc., \textbf{11} (2010), 347--383.

\bibitem{LodayRonco}
\bysame, Permutads, Preprint~\texttt{arXiv:1105.5271}.

\bibitem{Ma}
Toufik Mansour, \emph{Pattern avoidance in coloured permutations}, S\'em.
  Lothar. Combin. \textbf{46} (2001/02), Art. B46g, 12 pp. (electronic).

\bibitem{MRHopf}
Clauda Malvenuto and Christophe Reutenauer, \emph{Duality between quasi-symmetric functions and the Solomon descent algebra},
J. Algebra \textbf{177} (1995), no.~3, 967--982. 

\bibitem{MR}
Anthony Mendes and Jeffrey Remmel, \emph{Permutations and words counted by
  consecutive patterns}, Adv. in Appl. Math. \textbf{37} (2006), no.~4,
  443--480.

\bibitem{N} 
Brian Nakamura, \emph{Computational approaches to consecutive pattern avoidance in permutations},
Preprint \texttt{arXiv:1102.2480}.

\bibitem{NZ}
John Noonan and Doron Zeilberger, \emph{The {G}oulden-{J}ackson cluster method:
  extensions, applications and implementations}, J. Differ. Equations Appl.
  \textbf{5} (1999), no.~4-5, 355--377.

\bibitem{P} Dmitri Piontkovskii, \emph{On the growth of graded algebras with a small number of defining relations}, (Russian) Uspekhi Mat. Nauk 48 (1993), no. 3(291), 199--200; translation in 
Russian Math. Surveys \textbf{48} (1993), no. 3, 211--212. 

\bibitem{GShApp4} Narad Rampersad, \emph{Further applications of a power series method for
pattern avoidance}, Electronic J. Comb. \textbf{18(1)}, 2011, article P134, 8~pp. 

\bibitem{Ronco}
Maria Ronco, \emph{Shuffle bialgebras}, Ann. Inst. Fourier, to appear.

\bibitem{Stover} 
Christopher R.~Stover, \emph{The equivalence of certain categories of twisted Lie and Hopf algebras over a commutative ring.}
J. Pure Appl. Algebra \textbf{86} (1993), no.~3, 289--326. 

\bibitem{oeis}
Neil J.~A. Sloane, \emph{On-line encyclopedia of integer sequences}, Available
  via the URL \texttt{http://www.research.att.com/$\sim$njas/sequences/}.

\bibitem{Specht}
Wilhelm Specht, \emph{Gesetze in {R}ingen, {I}}, Math. Z. \textbf{52} (1950),
  557--589.

\bibitem{Stan}
Richard~P. Stanley, \emph{Enumerative combinatorics. {V}ol. 2}, Cambridge
  Studies in Advanced Mathematics, vol.~62, Cambridge University Press,
  Cambridge, 1999, With a foreword by Gian-Carlo Rota and appendix 1 by Sergey
  Fomin.

\bibitem{Stein}
Einar Steingr\'imsson, \emph{Generalized permutation patterns -- a short
  survey}, Permutation Patterns, St Andrews 2007, LMS Lecture Note Series, vol.
  376, Cambridge University Press, Cambridge, 2010.

\bibitem{Ufn}
Victor~A. Ufnarovskij, \emph{Combinatorial and asymptotic methods in algebra},
  Algebra, {VI}, Encyclopaedia Math. Sci., vol.~57, Springer, Berlin, 1995,
  pp.~1--196.

\bibitem{Wilf}
Herbert~S. Wilf, \emph{The patterns of permutations}, Discrete Math.
  \textbf{257} (2002), no.~2-3, 575--583.

\end{thebibliography}

\providecommand{\bysame}{\leavevmode\hbox to3em{\hrulefill}\thinspace}
\providecommand{\MR}{\relax\ifhmode\unskip\space\fi MR }
\providecommand{\MRhref}[2]{%
  \href{http://www.ams.org/mathscinet-getitem?mr=#1}{#2}
}
\providecommand{\href}[2]{#2}

\end{document}